\documentclass[twoside,reqno,A4]{amsart}

\usepackage{latexsym,rotate,eucal,cite}
\usepackage{amsmath,amsthm,amssymb,amsxtra}
\usepackage{cite}

\newtheorem{theorem}{Theorem}[section]
\newtheorem{lemma}[theorem]{Lemma}

\theoremstyle{definition}

\theoremstyle{remark}
\newtheorem{remark}[theorem]{Remark}

\numberwithin{equation}{section}

\begin{document} 

\title{Biharmonic hypersurfaces with three distinct principal curvatures in spheres}
\dedicatory{\textsc{
    Yu Fu \\[1mm]}
School of Mathematics and
Quantitative Economics\\
Dongbei University of Finance and Economics\\
Dalian 116025, P. R. China}
\thanks{Email addresses: yufudufe@gmail.com}


\begin{abstract}
We obtain a complete classification of proper biharmonic
hypersurfaces with at most three distinct principal curvatures in
sphere spaces with arbitrary dimension. Precisely, together with
known results of Balmu\c{s}-Montaldo-Oniciuc, we prove that compact
orientable proper biharmonic hypersurfaces with at most three
distinct principal curvatures in sphere spaces $\mathbb S^{n+1}$ are
either the hypersphere $\mathbb S^n(1/\sqrt2)$ or the Clifford
hypersurface $\mathbb S^{n_1}(1/\sqrt2)\times\mathbb
S^{n_2}(1/\sqrt2)$ with $n_1+n_2=n$ and $n_1\neq n_2$. Moreover, we
also show that there does not exist proper biharmonic hypersurface
with at most three distinct principal curvatures in hyperbolic
spaces $\mathbb H^{n+1}$.
\end{abstract}

\maketitle
\markboth{Yu Fu}{Biharmonic hypersurfaces in spheres}
\thispagestyle{empty}

\section{Introduction}
\hspace*{\parindent}
It is well known that the theory of harmonic maps plays a central
roles in various fields in differential geometry. The harmonic maps
between two Riemannian manifolds are critical points of the energy
functional
\begin{eqnarray*}
E(\phi)=\frac{1}{2}\int_M|d\phi|^2v_g
\end{eqnarray*}
for smooth maps $\phi:
(M^n,g)\longrightarrow(\bar{M}^m,\langle,\rangle)$.

Biharmonic maps
$\phi:(M^n,g)\longrightarrow(\bar{M}^m,\langle,\rangle)$ between
Riemannian manifolds are critical points of the bienergy functional
\begin{eqnarray*}
E_2(\phi)=\frac{1}{2}\int_M|\tau(\phi)|^2v_g,
\end{eqnarray*}
where $\tau(\phi)= {\rm trace \nabla d\phi}$ is the tension field of
$\phi$ that vanishes for harmonic maps. For biharmonic map, the
bitension field satisfies the associated Euler-Lagrange equation
(see \cite{jiang1986})
\begin{eqnarray*}
\tau_2(\phi)=-\Delta\tau(\phi)-{\rm trace}
R^{\bar{M}}(d\phi,\tau(\phi))d\phi=0,
\end{eqnarray*}
where $R^{\bar{M}}$ is the curvature tensor
\begin{eqnarray*}
R^{\bar{M}}(U,V)=\nabla_U^{\bar{M}}\nabla_V^{\bar{M}}-
\nabla_V^{\bar{M}}\nabla_U^{\bar{M}}-\nabla_{[U, V]}^{\bar{M}},\quad
U, V\in X(\bar{M}),
\end{eqnarray*}
and $\Delta$ is the rough Laplacian given by
\begin{eqnarray*}
\Delta=-\sum_{k=1}^n(\nabla^{\phi}_{e_k}\nabla^{\phi}_{e_k}-\nabla^{\phi}_{\nabla^{M}_{e_k}e_k})
\end{eqnarray*}
for a local orthonormal frame field $\{e_k\}_{k=1}^n$ defined on
$(M^n, g)$.

The above equation shows that $\phi$ is a biharmonic map if and only
if its bi-tension field $\tau_2(\phi)$ vanishes. Equivalently, for
an immersion
$\phi:(M^n,g)\longrightarrow(\bar{M}^m,\langle,\rangle)$ between
Riemannian manifolds, the mean curvature vector field
$\overrightarrow{H}$ satisfies the following fourth order elliptic
semi-linear PDE
\begin{eqnarray}
\Delta\overrightarrow{H}+{\rm trace}\,
R^{\bar{M}}(d\phi,\overrightarrow{H})d\phi=0.
\end{eqnarray}
Obviously, any minimal immersion, i.e. immersion satisfying
$\overrightarrow{H}=0$, is biharmonic. The non-harmonic biharmonic
immersions are called proper biharmonic.

In a different setting, B. Y. Chen in the middle of 1980s initiated
the study of biharmonic submanifolds in a Euclidean space by the
condition $\Delta \overrightarrow{H}=0$, where $\Delta$ is the rough
Laplacian of submanifolds with respect to the induced metric. It is
easy to see that both notions of biharmonic submanifolds in
Euclidean spaces coincide with each other.

The study of biharmonic submanifolds is nowadays a very active
subject. There is a challenging biharmonic conjecture of B. Y. Chen
made in 1991 \cite{Chen1991}:

{\bf Chen's conjecture}: {\em The only biharmonic submanifolds of
Euclidean spaces are the minimal ones}.

Ten years later, in 2001 Caddeo, Montaldo and Oniciuc \cite{CMO2001}
made the following generalized Chen's conjecture:

{\bf Generalized Chen's conjecture}: {\em Every biharmonic
submanifold of a Riemannian manifold with non-positive sectional
curvature is minimal}.

Recently, the Generalized Chen's conjecture was proved to be wrong
by Y. L. Ou and L. Tang in \cite{Ou2012}, who constructed examples
of proper-biharmonic hypersurfaces in a 5-dimensional space of
non-constant negative sectional curvature. However, the original
Chen's conjecture is still open so far. Also, the Generalized Chen's
conjecture is still open in its full generality for ambient spaces
with constant non-positive sectional curvature. For more recent
developments of Chen's conjecture and Generalized Chen's conjecture,
please refer to Chen's recent survey article \cite{chen2013} and
reference therein.

In contrast, the class of proper biharmonic submanifolds in sphere
spaces is rather rich and very interesting. The complete
classifications of biharmonic hypersurfaces in $\mathbb S^3$ and
$\mathbb S^4$ were obtained by Balmu\c{s}, Caddeo, Montaldo and
Oniciuc in \cite{CMO2001, BMO20102}. Moreover, the authors in
\cite{BMO2008} classified biharmonic hypersurfaces with at most two
distinct principal curvatures in $\mathbb S^n$ with arbitrary
dimension. There are also some results on biharmonic submanifolds in
general ambient space, e.g. \cite{Ou2010}.

For what concerns biharmonic hypersurfaces with three distinct
principal curvatures in spheres, Balmu\c{s}-Montaldo-Oniciuc in
\cite{BMO20102} proved the following non-existence result: there do
not exist compact constant mean curvature (CMC) proper-biharmonic
hypersurfaces with three distinct principal curvatures in $\mathbb
S^n$ everywhere.

In the present paper, we concentrate on biharmonic hypersurfaces
with three distinct principal curvatures in space forms with
arbitrary dimension. Firstly, we prove that biharmonic hypersurface
$M^n$ with at most three distinct principal curvatures in space
forms necessarily has constant mean curvature. Combining with
Balmu\c{s} et al.'s nice work on this subject, we can achieve a
complete classification of compact proper biharmonic hypersurfaces
with at most three distinct principal curvatures in spheres with
arbitrary dimension, and without any other assumptions. Hence, our
results extend all the known results mentioned above for biharmonic
hypersurfaces in spheres. At last, with a similar argument we also
show that there does not exist proper biharmonic hypersurface with
at most three distinct principal curvatures in hyperbolic spaces
$\mathbb H^{n+1}$.

\section{Preliminaries}
Let $M^n$ be an orientable hypersurface isometrically immersed into
a space form $R^{n+1}(c)$ with constant sectional curvature $c$.
Denote the Levi-Civita connections of $M^n$ and $R^{n+1}(c)$ by
$\nabla$ and $\tilde\nabla$, respectively. Let $X$ and $Y$ denote
vector fields tangent to $M^n$ and let $\xi$ be a unite normal
vector field. Then the Gauss and Weingarten formulas (cf.
\cite{chenbook2011, chenbook1984}) are given, respectively, by
\begin{eqnarray}
\tilde\nabla_XY&=&\nabla_XY+h(X,Y),\label{l23}\\
\tilde\nabla_X\xi&=&-AX,\label{l16}
\end{eqnarray}
where $h$ is the second fundamental form, and $A$ is the Weingarten
operator. It is well known that the second fundamental form $h$ and
the Weingarten operator $A$ are related by
\begin{eqnarray}\label{l3}
\langle h(X,Y),\xi\rangle=\langle AX,Y\rangle.
\end{eqnarray}
The mean curvature vector field $\overrightarrow{H}$ is given by
\begin{eqnarray}
\overrightarrow{H}=\frac{1}{n}{\rm trace}~h.
\end{eqnarray}
Moreover, the Gauss and Codazzi equations are given respectively by
\begin{eqnarray*}
R(X,Y)Z=c(\langle Y, Z\rangle X-\langle X,Z\rangle Y)+\langle
AY,Z\rangle AX-\langle AX,Z\rangle AY,
\end{eqnarray*}
\begin{eqnarray*}
(\nabla_{X} A)Y=(\nabla_{Y} A)X,
\end{eqnarray*}
where $R$ is the curvature tensor of hypersurface $M^n$ and
$(\nabla_XA)Y$ is defined by
\begin{eqnarray}\label{l7}
(\nabla_XA)Y=(\nabla_X)AY-A(\nabla_XY)
\end{eqnarray}
for all $X, Y, Z$ tangent to $M^n$.

Assume that $\overrightarrow{H}=H\xi$ and $H$ denotes the mean
curvature.

By identifying the tangent and the normal parts of the biharmonic
condition (1.1) for hypersurfaces in space forms $R^{n+1}(c)$, we
obtain the following characterization result for $M^n$ to be
biharmonic (see also \cite{CMO2002, BMO20102, chenbook1984}).
\begin{theorem}
The immersion $x: M^n\rightarrow R^{n+1}(c)$ of a hypersurface $M^n$
in an $n+1$-dimensional space form $R^{n+1}(c)$ is biharmonic if and
only if
\begin{equation}
\begin{cases}
\Delta H+H {\rm trace}\, A^2 =ncH,\\
2A\,{\rm grad}H+n\, H{\rm grad}H=0,
\end{cases}
\end{equation}
\end{theorem}
Clearly, it follows from (2.6) that the only umbilical proper
biharmonic hypersurface in $\mathbb S^{n+1}$ is an open part of
$\mathbb S^n(1/\sqrt2)$.

Recall those known results on biharmonic hypersurfaces with at most
two distinct principal curvatures in $\mathbb S^{n+1}$ developed by
Balmu\c{s} et al. in the last ten years.
\begin{theorem} {\rm([5])}
Let $M^2$ be a proper biharmonic surface in spheres $\mathbb S^3$.
Then $M^2$ is an open part of $\mathbb S^2(1/\sqrt2)\subset\mathbb
S^3$.
\end{theorem}
\begin{theorem} {\rm([3])}
Let $M^n$ be a proper biharmonic hypersurfaces with at most two
distinct principal curvatures in $\mathbb S^{n+1}$. Then $M^n$ is
either an open part of hypersphere $\mathbb S^n(1/\sqrt2)$ or
Clifford hypersurface $\mathbb S^{n_1}(1/\sqrt2)\times\mathbb
S^{n_2}(1/\sqrt2)$ with $n_1+n_2=n$ and $n_1\neq n_2$. Moreover, if
$M^n$ is complete, then either $M^n$ is the hypersphere $\mathbb
S^n(1/\sqrt2)$ or the Clifford hypersurface $\mathbb
S^{n_1}(1/\sqrt2)\times\mathbb S^{n_2}(1/\sqrt2)$ with $n_1+n_2=n$
and $n_1\neq n_2$.
\end{theorem}
For biharmonic hypersurfaces with at most three distinct principal
curvatures, Balmu\c{s} et al. obtained in \cite{BMO20102} the
following results.
\begin{theorem}
Let $M^3$ be a biharmonic hypersurface of the space form $\mathbb
E^4(c)$. Then $M^3$ has constant mean curvature.
\end{theorem}
\begin{theorem}
The only compact proper biharmonic hypersurfaces of $\mathbb S^4$
are the hypersphere $\mathbb S^3(1/\sqrt2)$ and the torus $\mathbb
S^1(1/\sqrt2)\times\mathbb S^2(1/\sqrt2)$.
\end{theorem}
\begin{theorem}
There exist no compact proper biharmonic hypersurfaces of constant
mean curvature and with three distinct principal curvatures in the
unit Euclidean spheres.
\end{theorem}
\section{Biharmonic hypersurfaces with three distinct principal curvatures in $ R^{n+1}(c)$}
We will concentrate on an orientable biharmonic hypersurface $M^n$
in a space form $R^{n+1}(c)$ with $n\geq4$. With the techniques
developed by B. Y. Chen in [10] (see also [4, 12-15]), we firstly
prove the following result.
\begin{theorem}
Let $M^n$ be an orientable proper biharmonic hypersurface with at
most three distinct principal curvatures in $R^{n+1}(c)$. Then $M^n$
has constant mean curvature.
\end{theorem}
It is known that the set $M_A$ of all points of $M$, at which the
number of distinct eigenvalues of the Weingarten operator $A$ (i.e.
the principal curvatures) is locally constant, is open and dense in
$M^n$. Therefore, as $M^n$ has at most three distinct principal
curvatures everywhere, one can work only on the connected component
of $M_A$ consisting by points where the number of principal
curvatures is three (it is already known that on the connected
components of $M_A$ where the number of distinct principal
curvatures is one or two, $M^n$ is CMC, i.e. the mean curvature is
constant; in the end, by passing to the limit, $H$ will be constant
on the whole $M^n$). On that connected component, the principal
curvature functions of $A$ are smooth.

We now suppose that, on the component, the mean curvature $H$ is not
constant. Thus, there is a point $x_0$ where $({\rm
grad}H)(x_0)\neq0$. In the following, we will work on an
neighborhood of $x_0$ where $({\rm grad}H)(x_0)\neq0$ at any point.

In view of the second equation of (2.6), we have that ${\rm
grad}\,H$ is an eigenvector of the Weingarten operator $A$ with the
corresponding principal curvature $-\frac{n}{2}H$. Without loss of
generality, we choose $e_1$ such that $e_1$ is parallel to ${\rm
grad}\,H$, and therefore the Weingarten operator $A$ of $M^n$ takes
the following form with respect to a suitable orthonormal frame
$\{e_1,\ldots, e_n\}$.
\begin{eqnarray}
A=\left( \begin{array}{cccc} \lambda_1&\\
&\lambda_2 &\\&&\ddots &\\&&&\lambda_n
\end{array} \right),
\end{eqnarray}
where $\lambda_i$ are the principal curvatures and
$\lambda_1=-\frac{n}{2}H$. Since $e_1$ is parallel to ${\rm
grad}\,H$, we compute
\begin{eqnarray*}
{\rm grad}H=\sum_{i=1}^ne_i(H)e_i
\end{eqnarray*}
and hence
\begin{eqnarray}
e_1(H)\neq0,\quad e_i(H)=0, \quad i=2, 3, \ldots, n.
\end{eqnarray}
We write
\begin{eqnarray}
\nabla_{e_i}e_j=\sum_{k=1}^n\omega_{ij}^ke_k,\quad i,j=1, 2, \ldots,
n.
\end{eqnarray}
We compute the compatibility conditions $\nabla_{e_k}\langle
e_i,e_i\rangle=0$ and $\nabla_{e_k}\langle e_i,e_j\rangle=0$, which
imply respectively that
\begin{eqnarray}
\omega_{ki}^i=0,\quad \omega_{ki}^j+\omega_{kj}^i=0,
\end{eqnarray}
for $i\neq j$ and $i, j, k=1, 2, \ldots, n$. Furthermore, we deduce
from (3.1) and (3.3) and the Codazzi equation that
\begin{eqnarray}
e_i(\lambda_j)=(\lambda_i-\lambda_j)\omega_{ji}^j,\\
(\lambda_i-\lambda_j)\omega_{ki}^j=(\lambda_k-\lambda_j)\omega_{ik}^j
\end{eqnarray}
for distinct $i, j, k=1, 2, \ldots, n$.

It follows from (3.2) and (3.3) that
\begin{eqnarray*}
[e_i,e_j](H)=0,\quad i, j=2, 3, \ldots, n, \quad i\neq j,
\end{eqnarray*}
which yields
\begin{eqnarray}
\omega_{ij}^1=\omega_{ji}^1,
\end{eqnarray}
for distinct $i, j=2, 3, \ldots, n$.

We claim that $\lambda_j\neq\lambda_1$ for $j=2, 3, \ldots, n$. In
fact, if $\lambda_j=\lambda_1$ for $j\neq1$, by putting $i=1$ in
(3.5) we have that
\begin{eqnarray}
0=(\lambda_1-\lambda_j)\omega_{j1}^j=e_1(\lambda_j)=e_1(\lambda_1),
\end{eqnarray}
which contradicts the first expression of (3.2).

By the assumption, $M^n$ is a nondegenerate hypersurface with three
distinct principal curvatures. Without loss of generality, we assume
that
\begin{eqnarray}
&&\lambda_2=\lambda_3=\ldots=\lambda_p=\alpha,\nonumber\\
&&\lambda_{p+1}=\lambda_{p+2}=\ldots=\lambda_n=\beta\nonumber
\end{eqnarray}
for $\frac{n+1}{2}\leq p<n$. The multiplicities of principal
curvatures $\alpha$ and $\beta$ are $p-1$ and $n-p$, respectively.

By the definition (2.4) of $\overrightarrow{H}$, we have
$nH=\sum_{i=1}^n\lambda_i$. Hence
\begin{eqnarray}
\beta=\frac{\frac{3}{2}nH-(p-1)\alpha}{n-p}.
\end{eqnarray}
Since $\lambda_j\neq\lambda_1$ for $j=2,\ldots,n$, we obtain
\begin{eqnarray}
\alpha\neq-\frac{n}{2}H, ~~\frac{3n}{2(n-1)}H,~~
\frac{n^2-(p-3)n}{2(p-1)}H.
\end{eqnarray}
We will derive some information from (3.5).

Since $n\geq4$, it follows from (3.9) that $p-1\geq2$. For $i, j=2,
3,\ldots,p$ and $i\neq j$ in (3.5), one has
\begin{eqnarray}
e_i(\alpha)=0, \quad i=2, 3, \ldots, p.
\end{eqnarray}
Depending on the multiplicity $n-p$ of the principal curvature
$\beta$, we consider two cases: \newline{\bf Case A}: $n-p\geq2$. In
this case, for $i, j=p+1, \ldots,n$ and $i\neq j$ in (3.5) we have
\begin{eqnarray}
e_i(\beta)=0, \quad i=p+1, \ldots, n.
\end{eqnarray}
Hence, it follows directly from (3.2), (3.9), (3.11) and (3.12) that
\begin{eqnarray}
e_i(\alpha)=0, \quad i=2, \ldots, n.
\end{eqnarray}
{\bf Case B}: $n-p=1$. Then (3.11) reduces to
\begin{eqnarray}
e_i(\alpha)=0, \quad i=2, \ldots, n-1.
\end{eqnarray}
In this case, we will show that $e_n(\alpha)=0$ in the following.

Let us compute
$[e_1,e_i](H)=\big(\nabla_{e_1}e_i-\nabla_{e_i}e_1\big)(H)$ for
$i=2, \ldots, n$. From the first expression of (3.4), we have
$\omega_{i1}^1=0$. For $j=1$ and $i\neq1$ in (3.5), by (3.2) we have
$\omega_{1i}^1=0$ $(i\neq1)$. Hence we have
\begin{eqnarray}
e_ie_1(H)=0,\quad i=2, \ldots, n.
\end{eqnarray}
By (3.14), with a similar way we can show that
\begin{eqnarray}
e_ie_1(\alpha)=0,\quad i=2, \ldots, n-1.
\end{eqnarray}

 For $j=1$, $k, i\neq1$ in (3.6) we have
\begin{eqnarray*}
(\lambda_i-\lambda_1)\omega_{ki}^1=(\lambda_k-\lambda_1)\omega_{ik}^1,
\end{eqnarray*}
which together with (3.7) yields
\begin{eqnarray}
\omega_{ij}^1=0, \quad i\neq j,\quad i, j=2,\ldots n.
\end{eqnarray}
Combining (3.17) with the second equation of (3.4) gives
\begin{eqnarray}
\omega_{i1}^j=0, \quad i\neq j,\quad i, j=2,\ldots n.
\end{eqnarray}
It follows from (3.5) that
\begin{eqnarray}
\omega_{i1}^i=\frac{e_1(\lambda_i)}{\lambda_1-\lambda_i}, \quad
i=2,\ldots n.
\end{eqnarray}
For $k=2$ and $i=n$ in (3.6), we have
\begin{eqnarray*}
(\lambda_n-\lambda_j)\omega_{2n}^j=(\lambda_2-\lambda_j)\omega_{n2}^j,
\end{eqnarray*}
which yields
\begin{eqnarray*}
\omega_{2n}^j=0, \quad j=3,\ldots n-1.
\end{eqnarray*}
Hence, from the first expression of (3.4) and (3.17) we get
\begin{eqnarray}
\omega_{2n}^j=0, \quad j=1, 3,\ldots n.
\end{eqnarray}
Also, (3.5) yields
\begin{eqnarray}
\omega_{2n}^2=\frac{e_n(\alpha)}{\lambda_n-\alpha}.
\end{eqnarray}
In the following we will derive a useful equation.

From the Gauss equation and (3.1) we have $R(e_2,e_n)e_1=0$. Recall
the definition of Gauss curvature tensor
\begin{eqnarray*}
R(X,Y)Z=\nabla_X\nabla_YZ-\nabla_Y\nabla_XZ-\nabla_{[X,Y]}Z.
\end{eqnarray*}
It follows from (3.16), (3.18-21) and (3.4) that
\begin{eqnarray*}
&&\nabla_{e_2}\nabla_{e_n}e_1=\frac{e_1(\lambda_n)e_n(\alpha)}{(\lambda_1-\lambda_n)(\lambda_n-\alpha)}e_2,\\
&&\nabla_{e_n}\nabla_{e_2}e_1=e_n(\frac{e_1(\alpha)}{\lambda_1-\alpha})e_2+\frac{e_1(\alpha)}{\lambda_1-\alpha}\sum_{k=3}^{n}\omega_{n2}^ke_k,\\
&&\nabla_{[e_2,e_n]}e_1=\frac{e_n(\alpha)e_1(\alpha)}{(\lambda_n-\alpha)(\lambda_1-\alpha)}
e_2-\frac{e_1(\alpha)}{\lambda_1-\alpha}\sum_{k=3}^{n}\omega_{n2}^ke_k.
\end{eqnarray*}
Hence
\begin{eqnarray}
e_n(\frac{e_1(\alpha)}{\lambda_1-\alpha})=\frac{e_1(\lambda_n)e_n(\alpha)}{(\lambda_1-\lambda_n)(\lambda_n-\alpha)}-
\frac{e_n(\alpha)e_1(\alpha)}{(\lambda_n-\alpha)(\lambda_1-\alpha)}.
\end{eqnarray}
Note that $\lambda_1=-\frac{n}{2}H$ and
$\lambda_n=\beta=\frac{3}{2}nH-(n-2)\alpha$ in this case.

It follows from (3.5) that
\begin{eqnarray}
\omega_{ii}^1=-\omega_{i1}^i=-\frac{e_1(\lambda_i)}{\lambda_1-\lambda_i}.
\end{eqnarray}
Consider the first equation of biharmonic equations (2.6). It
follows from (3.1) and (3.19) that
\begin{equation}
-e_1e_1(H)+\big(\frac{(n-2)e_1(\alpha)}{\lambda_1-\alpha}-\frac{e_1(\lambda_n)}{\lambda_1-\lambda_n}\big)e_1(H)+H[{\lambda_1}^2+(n-2)\alpha^2+{\lambda_n}^2]=ncH.
\end{equation}
Differentiating (3.24) along $e_n$, by (3.2), (3.15) and (3.22) we
get
\begin{equation*}
\frac{2}{\lambda_1-\lambda_n}\Big(\frac{e_1(\lambda_n)}{\lambda_1-\lambda_n}-\frac{\alpha}{\lambda_1-\alpha}\Big)e_1(H)e_n(\alpha)+H\big(-3nH+2(n-1)\alpha\big)e_n(\alpha)=0.
\end{equation*}
If $e_n(\alpha)\neq0$, then the above equation becomes
\begin{equation}
\frac{2}{\lambda_1-\lambda_n}\Big(\frac{e_1(\lambda_n)}{\lambda_1-\lambda_n}-\frac{\alpha}{\lambda_1-\alpha}\Big)e_1(H)+H\big(-3nH+2(n-1)\alpha\big)=0.
\end{equation}
Differentiating (3.25) along $e_n$, using (3.22) and (3.25) one has
\begin{eqnarray}
&&\frac{2n(4-n)H+2(n-2)(n-1)\alpha}{(\lambda_1-\lambda_n)(\lambda_n-\alpha)}\Big(\frac{e_1(\lambda_n)}{\lambda_1-\lambda_n}-\frac{\alpha}{\lambda_1-\alpha}\Big)e_1(H)\nonumber
\\&&+H\big((-7n+10)nH+4(n-1)(n-2)\alpha\big)=0.
\end{eqnarray}
Therefore, combining (3.26) with (3.25) gives
\begin{eqnarray*}
(n-2)H[3nH-2(n-1)\alpha]^2=0,
\end{eqnarray*}
which implies that
\begin{eqnarray*}
\alpha=\frac{3n}{2(n-1)}H.
\end{eqnarray*}
This contradicts (3.10). Hence, we have that $e_n(\alpha)=0$.

Now we are ready to express the connection coefficients of
hypersurfaces.
\begin{lemma}
Let $M^n$ be a biharmonic hypersurface with non-constant mean
curvature in spheres $\mathbb S^{n+1}$, whose shape operator given
by (3.1) with respect to an orthonormal frame $\{e_1, \ldots,
e_n\}$. Then we have
\begin{eqnarray*}
&&\nabla_{e_1}e_1=0;~
\nabla_{e_i}e_1=\frac{e_1(\lambda_i)}{\lambda_1-\lambda_i}e_i,~i=2,\ldots,n;\\
&&\nabla_{e_i}e_j=\sum_{k=2, k\neq j}^{p}\omega_{ij}^ke_k,
~i=1,\ldots,n,~j=2,\ldots,p,~i\neq j;\\
&&\nabla_{e_i}e_i=-\frac{e_1(\lambda_i)}{\lambda_1-\lambda_i}e_1+\sum_{k=2,
k\neq i}^{p}\omega_{ii}^ke_k,
~i=2,\ldots,p;\\
&&\nabla_{e_i}e_j=\sum_{k=p+1, k\neq j}^{n}\omega_{ij}^ke_k,
~i=1,\ldots,n,~j=p+1,\ldots,n,~i\neq j;\\
&&\nabla_{e_i}e_i=-\frac{e_1(\lambda_i)}{\lambda_1-\lambda_i}e_1+\sum_{k=p+1,
k\neq i}^{n}\omega_{ii}^ke_k, ~i=p+1,\ldots,n,
\end{eqnarray*}
where $\omega_{ki}^j=-\omega_{kj}^i$ for $i\neq j$ and $i, j,
k=1,\ldots, n$.
\end{lemma}
\begin{proof}
For $j=1$ and $i=2,\ldots,n$ in (3.5), by (3.2) we get
$\omega_{1i}^1=0$. Moreover, by the first and second expressions of
(3.4) we have
\begin{eqnarray}
\omega_{1i}^1=\omega_{11}^i=0,\quad i=1, \ldots, n.
\end{eqnarray}
For $i=1$, $j=2, \ldots, n$ in (3.5), we obtain
\begin{eqnarray}
\omega_{j1}^j=-\omega_{jj}^1=\frac{e_1(\lambda_j)}{\lambda_1-\lambda_j},
\quad j=2, \ldots, n.
\end{eqnarray}
For $i=p+1,\ldots,n$, $j=2,\ldots, p$ in (3.5), by (3.2) we have
\begin{eqnarray}
\omega_{ji}^j=-\omega_{jj}^i=0.
\end{eqnarray}
Similarly, for $i=2,\ldots, p$, $j=p+1,\ldots,n$ in (3.5), we also
have
\begin{eqnarray}
\omega_{ji}^j=-\omega_{jj}^i=0.
\end{eqnarray}
For $i=1$, by choosing $j, k=2,\ldots,p$ or $k, j=p+1,\ldots,n$
($j\neq k$) in (3.6), we have
\begin{eqnarray}
\omega_{k1}^j=\omega_{kj}^1=0.
\end{eqnarray}
For $i=2,\ldots, p$ and $j, k=p+1,\ldots, n$ ($j\neq k$) in (3.6),
we get
\begin{eqnarray}
\omega_{ki}^j=\omega_{kj}^i=0.
\end{eqnarray}
For $i=2,\ldots, p$, $j=1$ and $k=p+1,\ldots, n$ in (3.6), one has
\begin{eqnarray*}
(\alpha-\lambda_1)\omega_{ki}^1=(\beta-\lambda_1)\omega_{ik}^1,
\end{eqnarray*}
which together with (3.7) and the second expression of (3.4) gives
\begin{eqnarray}
\omega_{ki}^1=\omega_{ik}^1=\omega_{k1}^i=\omega_{i1}^k=0.
\end{eqnarray}
For $i=2,\ldots, p$, $k=1$ and $j=p+1,\ldots, n$ in (3.6), we obtain
\begin{eqnarray*}
(\beta-\alpha)\omega_{1i}^j=(\lambda_1-\alpha)\omega_{i1}^j,
\end{eqnarray*}
which together with (3.33) yields
\begin{eqnarray}
\omega_{1i}^j=\omega_{1j}^i=0.
\end{eqnarray}
Combining (3.27-3.34) with (3.4) completes the proof of the lemma.
\end{proof}
Define two smooth functions $A$ and $B$ as follows:
\begin{eqnarray}
A=\frac{e_1(\alpha)}{\lambda_1-\alpha},\quad
B=\frac{e_1(\beta)}{\lambda_1-\beta}.
\end{eqnarray}
One can compute the curvature tensor $R$ by Lemma 3.2, and apply the
Gauss equation for different values of $X$, $Y$ and $Z$. After
comparing the coefficients with respect to the orthonormal basis
$\{e_1, \ldots, e_n\}$ we get the following:
\begin{itemize}
\item $X=e_1, Y=e_2, Z=e_1$,
\begin{eqnarray}
e_1(A)+A^2=-\lambda_1\alpha-c;
\end{eqnarray}
\item $X=e_1, Y=e_n, Z=e_1$,
\begin{eqnarray}
e_1(B)+B^2=-\lambda_1\beta-c;
\end{eqnarray}
\item $X=e_n, Y=e_2, Z=e_n$,
\begin{eqnarray}
AB=-\alpha\beta-c.
\end{eqnarray}
\end{itemize}
Note that (3.38) is obtained by comparing the coefficient of $e_2$
in the equation.

Compute the first equation of biharmonic equations (2.6) again. It
follows from (3.1) and Lemma 3.2 that
\begin{equation}
-e_1e_1(H)-[(p-1)A+(n-p)B]e_1(H)+H[\lambda^2_1+(p-1)\alpha^2+(n-p)\beta^2]=ncH.
\end{equation}
\begin{lemma}
The functions $A$ and $B$ are related by
\begin{eqnarray}
&&[(4-p)A+(3+p-n)B]e_1(H)+\frac{3n^2(n+6-p)}{4(n-p)}H^3\nonumber\\
&&-\frac{3n(n-2+4p)}{2(n-p)}H^2\alpha+\frac{3n(p-1)}{n-p}H\alpha^2-3c(n+1)H=0.
\end{eqnarray}
\end{lemma}
\begin{proof}
From (3.35), (3.36) and (3.37) respectively reduce to
\begin{eqnarray}
&&e_1e_1(\alpha)+2Ae_1(\alpha)-Ae_1(\lambda_1)+(\lambda_1\alpha+c)(\lambda_1-\alpha)=0,\\
&&e_1e_1(\beta)+2Be_1(\beta)-Be_1(\lambda_1)+(\lambda_1\beta+c)(\lambda_1-\beta)=0.
\end{eqnarray}
By (3.9), it follows from the second expression of (3.35) that
\begin{eqnarray}
e_1(\alpha)=\frac{n-p}{p-1}\big(\frac{3n}{2(n-p)}e_1(H)-e_1(\beta)\big),\nonumber\\
=\frac{3n}{2(p-1)}e_1(H)-\frac{n-p}{p-1}B(\lambda_1-\beta).
\end{eqnarray}
Similarly, we have
\begin{eqnarray}
e_1(\beta)=\frac{3n}{2(n-p)}e_1(H)-\frac{p-1}{n-p}A(\lambda_1-\alpha).
\end{eqnarray}
Substitute (3.9) into (3.42). Eliminating $e_1e_1(H)$ and
$e_1e_1(\alpha)$, from (3.38), (3.39) and (3.41-44) we obtain the
desired equation (3.40).
\end{proof}

By the second expression of (3.35) and (3.9), (3.44) reduces to
\begin{eqnarray}
e_1(H)=-\big[\frac{p-1}{3}H+\frac{2(p-1)}{3n}\alpha\big]A+\big[-\frac{n+3-p}{3}H+\frac{2(p-1)}{3n}\alpha\big]B.
\end{eqnarray}
Substituting (3.45) into (3.40), by (3.38) we have
\begin{eqnarray}
&&(4-p)(p-1)(nH+2\alpha)A^2+(3+p-n)[n(n+3-p)H-2(p-1)\alpha]B^2\nonumber\\
&&=\frac{9n^3(n+6-p)}{4(n-p)}H^3+\frac{3n^2(p-1)(2p-2n-15)}{2(n-p)}H^2\alpha\nonumber\\
&&+\frac{n(p-1)(-2p^2+2pn+11p+n-12)}{n-p}H\alpha^2-\frac{2(p-1)^2(2p-n-1)}{n-p}\alpha^3\nonumber\\
&&+c(2p^2-5p-2np-4n)H+2c(p-1)(2p-n-1)\alpha.
\end{eqnarray}
Multiplying $A$ and $B$ successively on the equation (3.40), using
(3.38) one gets respectively
\begin{eqnarray}
&&(4-p)A^2e_1(H)=(3+p-n)(\alpha\beta+c) e_1(H)\\
&&+\Big[\frac{3n^2(n+6-p)}{4(n-p)}H^3-\frac{3n(n-2+4p)}{2(n-p)}H^2\alpha+\frac{3n(p-1)}{n-p}H\alpha^2-3c(n+1)H\Big]A=0.\nonumber\\
&&(3+p-n)B^2e_1(H)=(4-p)(\alpha\beta+c) e_1(H)\\
&&+\Big[\frac{3n^2(n+6-p)}{4(n-p)}H^3-\frac{3n(n-2+4p)}{2(n-p)}H^2\alpha+\frac{3n(p-1)}{n-p}H\alpha^2-3c(n+1)H\Big]B=0.\nonumber
\end{eqnarray}
Differentiating (3.40) along $e_1$, and using (3.36-37) and (3.39)
we get
\begin{eqnarray}
&&\Big[(4-p)(\frac{n}{2}H\alpha-A^2-c)+(3+p-n)(\frac{n}{2}H\beta-B^2-c)\Big]e_1(H)\nonumber\\
&&-\Big[(4-p)A+(3+p-n)B\Big]\Big[(p-1)A+(n-p)B\Big]e_1(H)\nonumber\\
&&+\Big[(4-p)A+(3+p-n)B\Big]\Big[\frac{n^2}{4}H^3+(p-1)H\alpha^2+(n-p)H\beta^2-ncH\Big]\nonumber\\
&&+\Big[\frac{9n^2(n+6-p)}{4(n-p)}H^2-\frac{3n(n-2+4p)}{n-p}H\alpha+\frac{3n(p-1)}{n-p}\alpha^2-3c(n+1)\Big]e_1(H)\nonumber\\
&&-\frac{3n(n-2+4p)}{2(n-p)}H^2e_1(\alpha)+\frac{6n(p-1)}{n-p}H\alpha
e_1(\alpha)=0.
\end{eqnarray}
Substituting (3.47), (3.48), (3.40) into (3.49), and using the first
expression of (3.35) we obtain
\begin{eqnarray*}
&&\Big[\frac{3n^2(2n-2p+21)}{4(n-p)}H^2-\frac{3n(5p+1)}{n-p}H\alpha+\frac{(p-1)(2n+7)}{n-p}\alpha^2-2c(n+5)\Big]e_1(H)\nonumber\\
&&+\Big[\frac{n^2(2pn-2p^2+7n+17p+30)}{4(n-p)}H^3-\frac{3n(3np+2p^2+4p-3n-6)}{2(n-p)}H^2\alpha\nonumber\\
&&+\frac{(p-1)(2np-2n+p-4)}{n-p}H\alpha^2-c(2pn+3p+4n)H\Big]A\nonumber\\
&&+\Big[\frac{n^2\big(2(n-p)^2+15(n-p)+45\big)}{4(n-p)}H^3-\frac{3n(n^2+np-2p^2+10p+n-8)}{2(n-p)}H^2\alpha\nonumber\\
&&+\frac{(p-1)(2n^2-2np+7n-p-3)}{n-p}H\alpha^2-c(2n^2-2np+9n-3p+3)H\Big]B=0.
\end{eqnarray*}
Moreover, it follows from (3.45) that the above equation further
reduces to
\begin{eqnarray}
&&\Big[\frac{9}{4}n^3(3n-2p+17)H^3-\frac{3}{2}n^2(-6p^2+11np+43p-11n-37)H^2\alpha\\
&&+n(p-1)(4np-4n+26p+1)H\alpha^2-2(p-1)^2(2n+7)\alpha^3\nonumber\\
&&+cn(n-p)(-4np-14n+7p-10)H+4c(n-p)(n+5)(p-1)\alpha\Big]A\nonumber\\
&&-\Big[\frac{9}{2}(2n-2p+3)H^3+\frac{9}{2}n^2(2p^2+n^2-3np-7p+n-3)H^2\alpha\nonumber\\
&&-2n(p-1)(2n^2-2np+4n-13p-18)H\alpha^2-2(p-1)^2(2n+7)\alpha^3\nonumber\\
&&+cn(n-p)(-4n^2+4np-11n+9p+21)H-4c(n-p)(n+5)(p-1)\alpha\Big]B=0.\nonumber
\end{eqnarray}
Now all the desired equations (3.38), (3.46) and (3.50) concerning
$A$ and $B$ are obtained.

In order to write handily, we introduce several notions: $L, M$
denoting the coefficients of $A$ and $B$ respectively in (3.50), and
$N$ denoting the right hand side of quality in equation (3.46). Then
(3.46) and (3.50) become
\begin{eqnarray}
&&(4-p)(p-1)(nH+2\alpha)A^2\\
&&+(3+p-n)[n(n+3-p)H-2(p-1)\alpha]B^2=N,\nonumber\\
&&LA-MB=0.
\end{eqnarray}
Multiplying $LM$ on the equation (3.51), using (3.52) and (3.38) we
can eliminate both $A$ and $B$. Hence, we have
\begin{eqnarray}
&&(4-p)(p-1)(nH+2\alpha)M^2\big(\frac{\frac{3}{2}nH\alpha-(p-1)\alpha^2}{n-p}+c\big)\nonumber\\
&&+(3+p-n)[n(n+3-p)H-2(p-1)\alpha]L^2\big(\frac{\frac{3}{2}nH\alpha-(p-1)\alpha^2}{n-p}+c\big)\nonumber\\
&&+LMN=0.
\end{eqnarray}
In view of (3.53), we note that the equation should take the
following form:
\begin{eqnarray}
&&a_{90}H^9+a_{81}H^8\alpha+a_{72}H^7\alpha^2+a_{63}H^6\alpha^3+a_{54}H^5\alpha^4
+a_{45}H^4\alpha^5\nonumber\\
&&+a_{36}H^3\alpha^6+a_{27}H^2\alpha^7+a_{18}H\alpha^8+a_{09}\alpha^9+c(a_{70}H^7+a_{61}H^6\alpha\nonumber\\
&&+a_{52}H^5\alpha^2+a_{43}H^4\alpha^3+a_{34}H^3\alpha^4+a_{25}H^2\alpha^5+a_{16}H\alpha^6+a_{07}\alpha^7\nonumber\\
&&+a_{50}H^5+a_{41}H^4\alpha+a_{32}H^3\alpha^2+a_{23}H^2\alpha^3+a_{14}H\alpha^5+a_{05}\alpha^5\nonumber\\
&&+a_{30}H^3+a_{21}H^2\alpha+a_{12}H\alpha^2+a_{03}\alpha^3)=0,
\end{eqnarray}
where the coefficients $a_{ij}$ ($i, j=0,\ldots,9$) are constants
concerning $n$ and $p$.

From (3.53), (3.50) and (3.46), we compute $a_{90}$ and $a_{09}$ as
follows
\begin{eqnarray*}
a_{90}=\frac{729n^6(n-p+6)(3n-2p+17)(2n-2p+3)}{32(n-p)}, \quad
a_{09}=0.
\end{eqnarray*}
Since $n>p$, it is easy to see that $a_{90}\neq0$.

Note that $\alpha$ is not constant in general. In fact, if $\alpha$
is a constant, then (3.54) becomes an algebraic equation of $H$ with
constant coefficients. Thus, the real function $H$ satisfies a
polynomial equation $q(H)=0$ with constant coefficients, therefore
it must be a constant. The conclusion follows immediately.

Consider an integral curve of $e_1$ passing through $p=\gamma(t_0)$
as $\gamma(t), t\in I$. Since $e_i(H)=e_i(\alpha)=0$ for
$i=2,\ldots, n$ and $e_1(H), e_1(\alpha)\neq0$, we can assume
$t=t(\alpha)$ and $H=H(\alpha)$ in some neighborhood of
$\alpha_0=\alpha(t_0)$.

From the first expression of (3.35), (3.45) and (3.52), we have
\begin{eqnarray}
\frac{dH}{d\alpha}&=&\frac{dH}{dt}\frac{dt}{d\alpha}=\frac{e_1(H)}{e_1(\alpha)}\nonumber\\
&=&\frac{-(\frac{p-1}{3}H+\frac{2(p-1)}{3n}\alpha)A+(-\frac{n+3-p}{3}H+\frac{2(p-1)}{3n}\alpha)B}{(-\frac{n}{2}H-\alpha)A}\nonumber\\
&=&\frac{2(p-1)}{3n}+\frac{(-\frac{n+3-p}{3}H+\frac{2(p-1)}{3n}\alpha)B}{(-\frac{n}{2}H-\alpha)A}\nonumber\\
&=&\frac{2(p-1)}{3n}+\frac{2\big((n+3-p)H-2(p-1)\alpha\big)L}{3n(nH+2\alpha)M}.
\end{eqnarray}

Differentiating (3.54) with respect to $\alpha$ and substituting
$\frac{dH}{d\alpha}$ from (3.55), combining these with (3.50) we get
another algebraic equation of twelfth degree concerning $H$ and
$\alpha$
\begin{eqnarray}
&&b_{12,0}H^{12}+b_{11,1}H^{11}\alpha+b_{10,2}H^{10}\alpha^2+b_{93}H^9\alpha^3+b_{84}H^8\alpha^4
+b_{75}H^7\alpha^5\nonumber\\
&&+b_{66}H^6\alpha^6+b_{57}H^5\alpha^7+b_{48}H^4\alpha^8+b_{39}H^3\alpha^9+b_{2,10}H^2\alpha^{10}+b_{1,11}H\alpha^{11}\nonumber\\
&&+b_{0,12}\alpha^{12}+c(b_{10,0}H^{10}+b_{91}H^9\alpha+b_{82}H^8\alpha^2+b_{73}H^7\alpha^3+b_{64}H^6\alpha^4
\nonumber\\
&&+b_{55}H^5\alpha^5+b_{46}H^4\alpha^6+b_{37}H^3\alpha^7+b_{28}H^2\alpha^8+b_{19}H\alpha^9+b_{0,10}\alpha^{10}+b_{80}H^8\nonumber\\
&&+b_{71}H^7\alpha+b_{62}H^6\alpha^2+b_{53}H^5\alpha^3+b_{44}H^4\alpha^4+b_{35}H^3\alpha^5+b_{26}H^2\alpha^6+b_{17}H\alpha^7\nonumber\\
&&+b_{08}\alpha^8+b_{60}H^6+b_{51}H^5\alpha+b_{42}H^4\alpha^2+b_{33}H^3\alpha^3+b_{24}H^2\alpha^4+b_{15}H\alpha^5\nonumber\\
&&+b_{06}\alpha^6+b_{40}H^4+b_{31}H^3\alpha+b_{22}H^2\alpha^2+b_{13}H\alpha^3+b_{04}\alpha^4)=0,
\end{eqnarray}
where the coefficients $b_{ij}$ ($i, j=0,\ldots,12$) are constants
concerning $n$ and $p$.

Note that equation (3.56) is non-trivial and different from (3.54).

We rewrite (3.54) and (3.56) respectively in the following forms
\begin{eqnarray}
\sum_{i=0}^8q_i(H)\alpha^i=0,\qquad \sum_{j=0}^{12}{\bar
q}_j(H)\alpha^j=0,
\end{eqnarray}
where $q_i(H)$ and ${\bar q}_j(H)$ are polynomials concerning
function $H$.

We may eliminate $\alpha$ between the two equations of (3.57).
Multiplying ${\bar q}_{12}(H)\alpha^4$ and $q_8(H)$ respectively on
the first and second equations of (3.57), we obtain a new polynomial
equation of $\alpha$ with eleventh degree. Combining this equation
with the first equation of (3.57), we successively obtain a
polynomial equation of $\alpha$ with tenth degree. In a similar way,
by using the first equation of (3.57) and its consequences we are
able to gradually eliminate $\alpha$.

At last, we obtain a non-trivial algebraic polynomial equation of
$H$ with constant coefficients. Therefore, we conclude that the real
function $H$ must be a constant, which contradicts our original
assumption.

In summary, we have proved Theorem 3.1 as stated in the beginning
part in this section.

Now we present our main theorem in the following.
\begin{theorem}
Let $M^n$ be an orientable compact proper biharmonic hypersurface
with at most three distinct principal curvatures in $\mathbb
S^{n+1}$. Then $M^n$ is either the hypersphere $\mathbb
S^n(1/\sqrt2)$ or the Clifford hypersurface $\mathbb
S^{n_1}(1/\sqrt2)\times\mathbb S^{n_2}(1/\sqrt2)$ with $n_1+n_2=n$
and $n_1\neq n_2$.
\end{theorem}
\begin{proof}
We only need to deal with the case of proper biharmonic
hypersurfaces $M^n$ with three distinct principal curvatures in
$\mathbb S^{n+1}$. According to Theorem 3.1, $M^n$ has constant mean
curvature $H$. Hence, Theorem 2.6 impies that this case is
impossible, which together with Theorem 2.3 leads to the conclusion.
\end{proof}
\begin{remark}
Theorem 3.4 extends Balmu\c{s} et al's results of Theorems 2.3, 2.5
and 2.6 in [3] and [4].
\end{remark}
A result due to Oniciuc \cite{oniciuc2002} says that a CMC
biharmonic immersion in a space form $R^n(c)$ for $c\leq0$ is
minimal. Hence, combining this with Theorem 3.1 implies immediately
that
\begin{theorem}
There exist no proper biharmonic hypersurfaces with at most three
distinct principal curvatures in Euclidean space $\mathbb E^{n+1}$
or hyperbolic spaces $\mathbb H^{n+1}$ with arbitrary dimension.
\end{theorem}
\begin{remark}
T. Hasanis and T. Vlachos \cite{HasanisVlachos1995} proved that
there exists no proper biharmonic hypersurface in $\mathbb E^4$ (see
also \cite{defever1998}). And, it was proved recently by the author
in \cite{fu} that there exists no proper biharmonic hypersurface
with at most three distinct principal curvatures in Euclidean
spaces. Thus, Theorem 3.6 recovers all the results in
\cite{HasanisVlachos1995, defever1998, fu} and \cite{BMO20102} for
hyperfaces in $\mathbb H^4$.
\end{remark}
\begin{remark}
Note that Theorem 3.6 gives an affirmative partial answer to the
Generalized Chen's conjecture.
\end{remark}


\end{document}